\documentclass[12pt,a4paper]{amsart}

\usepackage[centertags]{amsmath}
\usepackage{amsfonts,ifthen,latexsym,amssymb,amsthm, amscd}

\usepackage[T1]{fontenc}
\usepackage[utf8]{inputenc}
\usepackage[american]{babel}
\usepackage{appendix}
\usepackage{environ}
\NewEnviron{problem}[1]{%
\smallskip\begin{center}\fbox{\parbox{3in}{%
    {\centering\scshape #1\par}%
    \parskip=1ex
    \everypar{\hangindent=1em}%
    \BODY
}}\end{center}\smallskip}

\usepackage[usenames,dvipsnames]{color}

\definecolor{labelkey}{gray}{.20}
\definecolor{refkey}{gray}{.20}
\definecolor{eqkey}{gray}{.20}

\allowdisplaybreaks

\usepackage[american]{babel} 

\usepackage[perpage]{footmisc}

\usepackage{graphicx}

\usepackage{url}
\usepackage{hyperref}
\hypersetup{colorlinks=true,citecolor=blue,filecolor=blue,linkcolor=blue,urlcolor=blue}

\def\mathclap#1{\text{\hbox to 0pt{\hss$\mathsurround=0pt#1$\hss}}}

\newcommand{\beq}{\begin{equation}}
\newcommand{\eeq}{\end{equation}}
\newtheorem{theorem}{Theorem}
\newtheorem*{theorem*}{Theorem}
\newtheorem{thm}{Theorem}
\newtheorem*{thm*}{Theorem}

\newtheorem{lemma}[theorem]{Lemma}
\newtheorem*{lemma*}{Lemma}
\newtheorem{proposition}[theorem]{Proposition}
\newtheorem{prop}[theorem]{Proposition}

\newtheorem{cory}[theorem]{Corollary}

\theoremstyle{remark}

\newtheorem{remarks}[theorem]{Remarks}

\newcommand{\al}{{\alpha}}

\newcommand{\Om}{{\Omega}}

\newcommand{\eps}{{\varepsilon}}
\newcommand{\de}{{\delta}}

\newcommand{\ga}{{\gamma}}
\newcommand{\Ga}{{\Gamma}}

\newcommand{\si}{{\sigma}}

\renewcommand{\phi}{{\varphi}}

\newcommand\aut{\operatorname{Aut}}

\newcommand\R{\mathbb R}
\newcommand\Q{\mathbb Q}
\newcommand\Z{\mathbb Z}
\newcommand\C{\mathbb C}
\newcommand\N{\mathbb N}

\newcommand\ZZ{\mathbf Z}

\newcommand{\wh}[1]{{\widehat {#1}}}
\newcommand{\cal}[1]{{\mathcal #1}}

\makeatletter
\newcommand*\if@single[3]{%
  \setbox0\hbox{${\mathaccent"0362{#1}}^H$}%
  \setbox2\hbox{${\mathaccent"0362{\kern0pt#1}}^H$}%
  \ifdim\ht0=\ht2 #3\else #2\fi
  }
\newcommand*\rel@kern[1]{\kern#1\dimexpr\macc@kerna}
\newcommand*\widebar[1]{\@ifnextchar^{{\wide@bar{#1}{0}}}{\wide@bar{#1}{1}}}
\newcommand*\wide@bar[2]{\if@single{#1}{\wide@bar@{#1}{#2}{1}}{\wide@bar@{#1}{#2}{2}}}
\newcommand*\wide@bar@[3]{%
  \begingroup
  \def\mathaccent##1##2{%
    \if#32 \let\macc@nucleus\first@char \fi
    \setbox\z@\hbox{$\macc@style{\macc@nucleus}_{}$}%
    \setbox\tw@\hbox{$\macc@style{\macc@nucleus}{}_{}$}%
    \dimen@\wd\tw@
    \advance\dimen@-\wd\z@
    \divide\dimen@ 3
    \@tempdima\wd\tw@
    \advance\@tempdima-\scriptspace
    \divide\@tempdima 10
    \advance\dimen@-\@tempdima
    \ifdim\dimen@>\z@ \dimen@0pt\fi
    \rel@kern{0.6}\kern-\dimen@
    \if#31
      \overline{\rel@kern{-0.6}\kern\dimen@\macc@nucleus\rel@kern{0.4}\kern\dimen@}%
      \advance\dimen@0.4\dimexpr\macc@kerna
      \let\final@kern#2%
      \ifdim\dimen@<\z@ \let\final@kern1\fi
      \if\final@kern1 \kern-\dimen@\fi
    \else
      \overline{\rel@kern{-0.6}\kern\dimen@#1}%
    \fi
  }%
  \macc@depth\@ne
  \let\math@bgroup\@empty \let\math@egroup\macc@set@skewchar
  \mathsurround\z@ \frozen@everymath{\mathgroup\macc@group\relax}%
  \macc@set@skewchar\relax
  \let\mathaccentV\macc@nested@a
  \if#31
    \macc@nested@a\relax111{#1}%
  \else
    \def\gobble@till@marker##1\endmarker{}%
    \futurelet\first@char\gobble@till@marker#1\endmarker
    \ifcat\noexpand\first@char A\else
      \def\first@char{}%
    \fi
    \macc@nested@a\relax111{\first@char}%
  \fi
  \endgroup
}
\makeatother

\newcommand*\mcapinn[2]{\vcenter{\hbox{$\mathsurround=0pt
  \ifx\displaystyle#1\textstyle\else#1\fi\bigcap$}}}

\newcommand*\mcupinn[2]{\vcenter{\hbox{$\mathsurround=0pt
  \ifx\displaystyle#1\textstyle\else#1\fi\bigcup$}}}

\newcommand\tr{\operatorname{tr}}

\newcommand\dimvn{\dim_\text{vN}}
\newcommand\trvn{\tr_\text{vN}}

\renewcommand{\ge}{\geqslant}

\newcommand{\wre}{\wr}
\newcommand{\lamplighter}{\ZZ_2\wr\ZZ}
\newcommand{\ring}[2]{{\mathbb #1 \!\left[ #2 \right]}}
\newcommand{\zring}[1]{{\ring{Z}{#1}}}
\newcommand{\qring}[1]{{\ring{Q}{#1}}}
\newcommand{\cring}[1]{{\ring{C}{#1}}}
\newcommand{\ckl}{\ZZ}

\begin{document}

\title[Vanishing of $l^2$-cohomology as a computational problem]
 {Vanishing of $l^2$-cohomology as a computational problem} 
\author{{\L}ukasz Grabowski}
\let\thefootnote\relax\footnote{\hspace{-18pt}\textit{Email:} \texttt{graboluk@gmail.com}}
\let\thefootnote\relax\footnote{\hspace{-18pt}{The author was supported by EPSRC at Imperial College London and Oxford University, by the EPSRC grant EP/K012045/1 at University of Warwick, and by Fondation Sciences Math\' ematiques de Paris during the program \textit{Random Walks and Asymptotic Geometry of Groups} at IHP.}}
\maketitle
\vspace{-30pt}
\begin{center}
\textit{University of Warwick, Mathematics Institute, Zeeman Building, Coventry, CV4 7AL, UK}\end{center}
\begin{abstract}
We show that it is impossible to algorithmically decide if the $l^2$-cohomology of the universal cover of a finite CW complex is trivial, even if we only consider complexes whose fundamental group is equal to the elementary amenable group $(\lamplighter)^3$. A corollary of the proof is that there is no algorithm which decides if an element of the integral group ring of the group $(\lamplighter)^4$ is a zero-divisor.  On the other hand, we show, assuming some standard conjectures, that such an algorithm exists for the integral group ring of any group with a decidable word problem and a bound on the sizes of finite subgroups. 


\end{abstract}

\maketitle



\section{Introduction}

The $l^2$-cohomology was introduced by Atiyah \cite{Atiyah1976}. Since then it has been studied  in numerous contexts, e.g. \cite{cheeger_gromov:l2_cohomology_and_group_cohomology}, \cite{gaboriau_invariants_l2_de_relations_dequivalence}, \cite{lott-lueck-l2-topological-invariants-of-3-manifolds}.  For a very readable introduction see \cite{Eckmann_intro}, and for more up-to-date information see \cite{Lueck:Big_book}. 

If $X$ is a CW-complex, the $l^2$-cohomology of $X$ is defined in the same way as the CW-cohomology, but using  $l^2$-cochains and $l^2$-coboundaries. If $X$ is a finite complex and $\pi_1(X) = G$, the deck transformation action extends to the action of $G$ on the $l^2$-cohomology groups of the universal cover of $X$. In this way $l^2$-cohomology groups become modules over the von Neumann algebra of $G$. The von Neumann dimensions of the $l^2$-cohomology groups are called \textit{$l^2$-Betti numbers}. 

If $X$ is a finite CW-complex, the phrase \textit{$l^2$-Betti numbers of $X$} is  a shorthand for \textit{$l^2$-Betti numbers  of the universal cover of $X$ with respect to the action of $\pi_1(X)$}. 
\smallskip

One popular research subject is the determination of possible values of the $l^2$-Betti numbers. It is useful to fix a group $G$, and consider the set $\cal C(G)$ of all possible $l^2$-Betti numbers of finite CW-complexes with fundamental group $G$. The determination of $\cal C(G)$ is called \textit{the Atiyah problem for $G$}, since Atiyah asked a related question in his foundational paper \cite{Atiyah1976}. 

When $G$ is torsion-free, then conjecturally $\cal C(G) = \N :=\{0,1,\ldots\}$. This statement is known as  \textit{the Atiyah conjecture for torsion-free groups}. Similarly, let a BFS-group be a group for which there is a bound on the size of finite subgroups. For a BFS-group $G$  one conjectures the existence of $k=k(G)\in \N$ such that $\cal C(G) = \{\frac{1}{k}, \frac{2}{k},\ldots\}$.

\smallskip
Linnell \cite{Linnel:Atiyah_conjecture_for_free_groups} showed that the Atiyah conjecture for torsion-free groups implies  that $\Z[G]$ embeds into a skew-field. Similarly, when $G$ is a BFS-group, the Atiyah conjecture implies that $\Z[G]$ embeds into the ring of $k(G)\times k(G)$ matrices over a skew-field. 

A lot of appeal of the Atiyah conjecture comes from this structural result. In fact, Linnell showed that the existence of a skew-field of so called \textit{affiliated operators} which contains $\Z[G]$  is equivalent to the Atiyah conjecture.

For some non-BFS groups it has been shown recently (\cite{arxiv:austin-2009}, \cite{arxiv:grabowski-2010}, \cite{arxiv:pichot_schick_zuk-2010}, \cite{arxiv:lehner_wagner-2010}, \cite{arxiv:grabowski-2010-2}) that $\cal C(G)$ contains irrational numbers. All examples so far contain $\ZZ_p\wre\ZZ$ as a subgroup for some $p \ge 2$. 

Our main result shows that groups which contain $(\ZZ_p\wre\ZZ)^3$  have a property very much antipodal to the existence of a skew-field of affiliated operators: the computational problem of determining if a matrix over $\Z[G]$ is invertible as an affiliated operator is undecidable (for the precise meaning of \textit{undecidable} see for example \cite{sipser_introduction_to_the_theory_of_computation}). 

\smallskip

We first state the main result in terms of $l^2$-cohomology. Consider the following computational problem. Its input consists of a number $n$ and the description of the gluing maps used to build a CW-complex $X$. The algorithm should decide if the $n$-th $l^2$-Betti number of $X$ is $0$. We call this computational problem \textit{Trivial-$l^2$-Betti-numbers} for $G$.
\begin{thm}\label{thm-main}
 Let $G$ be a finitely-presented  group which contains $\lamplighter$. Then the problem  Trivial-$l^2$-Betti-numbers for $G^3$ is undecidable.
\end{thm}
\begin{remarks}
 \begin{enumerate}
  \item Note that the problem of determining if the ordinary cohomology of a finite CW-complex vanishes is decidable, since it boils down to checking if the kernel of a finite-dimensional matrix with integer coefficients is trivial.
\item The above theorem remains valid for $\ZZ_p\wre\ZZ$ in place of $\ZZ_2\wre\ZZ$, but to keep notation simpler we deal only with $\lamplighter$. 
\end{enumerate}
\end{remarks}

In the actual proofs we never deal with CW complexes, only with the combinatorial Laplacians (see \cite{Eckmann_intro}). Combinatorial Laplacians are matrices over the group ring $\Z[G]$ which can be read off easily from the gluing maps. The $l^2$-cohomology of a CW-complex is isomorphic with the kernel of the combinatorial Laplacian of the suitable degree, seen as an operator on $l^2(G)^m$.

Conversely, given a matrix $M$ over $\Z[G]$, we can easily build a CW-complex $X$ with fundamental group $G$,  such that in some degree the $l^2$-cohomology of $X$ will be isomorphic to $\ker M$ (see \cite{Eckmann_intro} again).  This shows that Trivial-$l^2$-Betti-numbers  is equivalent to the following computational problem.
\begin{problem}{Kernel-over-$\Z[G]$}
\textbf{Input:} matrix $M\in M_k(\Z[G])$ for some $k$

\textbf{Problem:} Is $\ker M =\{0\}$? (here $M$ is considered as an operator $l^2(G)^k \to l^2(G)^k$)
\end{problem}

This problem no longer requires $G$ to be finitely presented. In view of the two preceding paragraphs, Theorem \ref{thm-main} is a corollary of the following.
\begin{thm}\label{thm-main-honestly}
Kernel-over-$\Z[G]$  is undecidable for $G = (\lamplighter)^3$.

\end{thm}

\begin{remarks}
 \begin{enumerate}
\item The general strategy is to adapt  techniques from \cite{arxiv:grabowski-2010} which allow for ``realizing Turing machines as matrices over $\Z[G]$.''

 \item The currently most interesting case of Kernel-over-$\Z[G]$ problem is when $G= \lamplighter$. It seems the methods of this article could perhaps show that for  $G= (\lamplighter)^2$ the problem is undecidable, but for $\lamplighter$  some new idea would be needed. In the proof, the exponent bigger than $1$ corresponds closely to the fact that the halting problem for Turing machines which are \textit{read-only and operate on more than one tape} is undecidable. The halting problem for read-only Turing machines which operate on one tape is decidable (\cite[Chapter 4]{sipser_introduction_to_the_theory_of_computation}) and so our approach breaks down completely for $\lamplighter$. 

\item In \cite{arxiv:grabowski-2010-2} it is shown that $\cal C(\lamplighter)$ contains transcendental numbers. However, the method is more complicated than for $\cal C((\lamplighter)^k)$ when $k\ge 2$, for the same reason as in the previous remark. See \cite{arxiv:grabowski-2010} for an example of a transcendental number in $\cal C((\lamplighter)^3)$ which directly uses a read-only Turing machine on $3$ tapes.

\item It would be interesting to know if algorithms for Trivial-$l^2$-Betti-numbers  can be constructed using some more geometric methods (i.e. without directly passing to combinatorial Laplacians), at least for some class of ``nice'' CW-complexes or manifolds.
\item
It could be interesting to consider other computational problems for the group ring $\zring{G}$. Note that $T\in \zring G $ is not an $l^2$-zero-divisor if and only if it is an invertible in the ring of affiliated operators (\cite{Lueck:Big_book}, Lemma 8.8, Chapter 8). One could similarly consider the computational problems of being invertible in the group ring $\zring G$ or in the von Neumann algebra $LG$. This last problem corresponds to the question of whether the property of having a spectral gap around $0$ is decidable. It would be particularly interesting to give sufficient conditions for the spectral gap decidability, similar to Proposition \ref{prop-main} below. 

A problem related to the decidability of the spectral gap is computing the operator norm of an element in $\zring{G}$. Given a  finitely generated group $G$ we could ask for an algorithm which takes as an input a rational number $q$ and an element $T\in \zring{G}$ and correctly answers whether the norm of $T$ (as an operator on $l^2(G)$) is equal, less than, or greater than $q$. 

Certain related computational problems are considered in \cite{fritz-netzer-thome-can-you-compute-the-operator-norm}.
\end{enumerate}
\end{remarks}

It turns out that the Atiyah conjecture for BFS-groups implies fairly easily that Kernel-over-$\Z[G]$ is decidable.

\begin{proposition}\label{prop-main}
 Let  $G$ be  a finitely-presented sofic BFS-group with decidable word problem for which the Atiyah conjecture holds.  Then there is an algorithm which solves  Kernel-over-$\Z[G]$. 
\end{proposition}

\begin{remarks}
\begin{enumerate}\item  The property of sofic groups we use is the existence of a  bound on the spectral density of $T\in \Z[G]$,  computable in terms of the support and coefficients of $T$. Discussion of sofic groups and derivation of such a bound is deferred to the Appendix.  At present there are no groups which are proven not to be sofic.

It would be interesting to derive a computable bound on the spectral densities without using the soficity assumption.  

\item  If $G$ is as above except for having a decidable word problem, the proof actually shows that Kernel-over-$\Z[G]$ is decidable by a Turing machine with an oracle for the word problem of $G$ (see \cite{sipser_introduction_to_the_theory_of_computation} for Turing machines with oracles).

\end{enumerate}
\end{remarks}

Consider the following computational problem.
\begin{problem}{Zero-divisors-in-$\Z[G]$}
\textbf{Input:} an element $T$ of $\Z[G]$

\textbf{Problem:} Is $T$ a zero-divisor in $\Z[G]$?
\end{problem}

For amenable groups, it is well known that a matrix $M\in M_k(\Z[G])$ is a zero-divisor in $M_k(\Z[G])$ if and only if $\ker M \neq \{0\}$ (see e.g. \cite{Elek_the_strong_approximation_conjecture_holds_for_amenable_groups} or \cite{Pape_a_short_proof_of_the_approximation_conjecture_for_amenable_groups} for a short proof). This, together with Theorem \ref{thm-main-honestly}, shows that the zero-divisor problem for matrices over $\Z[(\lamplighter)^3]$ is undecidable. We can get rid of matrices for the price of increasing the exponent by one.

\begin{cory}\label{cory-intro-zero-divisors-in-lamplighter-4}
 Zero-divisors-in-$\Z[(\lamplighter)^4]$ is undecidable.
\end{cory}

The proof of this corollary is presented in Section \ref{section-lamplighter4}. The essential part of it is embedding arbitrarily large matrices into the group ring $\Q[\lamplighter]$. 

Proposition \ref{prop-main} is proven in Section \ref{section-sofic-decidable}. In the same section we present some easy consequences of Proposition \ref{prop-main} concerning Zero-divisors-in-$\Z[G]$  for amenable groups.  We discuss  also relations between Zero-divisors-in-$\Z[G]$ and the  Kaplansky zero-divisor conjecture for torsion-free groups.

Section \ref{section_prelim_on_turing_machines} discusses the properties of Turing machines which we need in the proof of Theorem \ref{thm-main-honestly} in Section \ref{section_embed}.

\section{Notation and conventions}\label{section-notation} 

The infinite cyclic group is $\ckl{}$, the cyclic group of order $n$ is $\ckl{n}$, the rings of integers, rationals, reals, and complex numbers are respectively $\Z$, $\Q$, $\R$, $\C$.

Given a ring $R$, the ring of $k\times k$-matrices over $R$ is denoted by $M_k(R)$. A {\it trace} $\tau$ on $R$ is a function $\tau\colon R \to \C$ such that $\tau(ab)=\tau(ba)$.  If $R$ is a ${}^*$-ring of operators on a Hilbert space then we also require $\tau(T^*T)$ to be  a non-negative real number, for all $T\in R$. The standard trace (i.e. sum of diagonal elements) on $M_k(\C)$ is denoted by $\tr$.

Given a group $G$, the complex group ring and the Hilbert space of $l^2$-summable functions are denoted by $\cring G$  and $l^2 (G)$. The standard basis elements of $l^2 (G)$ are denoted by $\zeta_g$, $g\in G$. We have an action of $\C [G]$ on $l^2(G)$ by bounded linear operators: it is induced by the action of $G$ on $l^2G$ defined by $g\cdot \zeta_h = \zeta_{gh}$. The canonical trace on $\cring G$ is defined by $\tr_{vN} (A) := \langle A\zeta_e, \zeta_e\rangle$, where $e\in G$ is the neutral element. We call it the \textit{von Neumann trace}, although usually this name is used only after taking the closure of $\C[G]$ with respect to the weak topology.

If $R$ is a ${}^*$-ring of operators on a Hilbert space, then a trace $\tau$ on $R$ is \textit{normal} if it extends to a continuous trace on the weak closure  of $R$. It is \textit{faithful} if, for every $T$, $\tau(T^*T)=0$ implies $T=0$. The traces defined above are faithful and normal.

If $\tau$ is a faithful normal trace on $R$, and $T^*=T\in R$ then the {\it spectral measure $\mu_T$ of $T$} is the usual projection-valued spectral measure of $T$ composed with $\tau$ (it makes sense to evaluate $\tau$ on spectral projections of $T$, since the latter are in the weak closure of $R$). The {\it von Neumann dimension of the kernel of $T$} is defined as $\dimvn \ker (T):=  \mu_T(\{0\}) $. For a general $T$ we define $\dimvn \ker (T) := \dimvn \ker (T^*T)$. 

The symbol $\tau$ also denotes the induced trace on $M_k(R)$, i.e. if $T\in M_k(R)$ is a matrix with entries $T_{ij}$ then $\tau(T) := \sum \tau(T_{ii})$. The spectral measure $\mu_T$ is computed with respect to this trace.

For more information on the spectral measures see \cite{Reed_Simon_Methods_of_moder_mathematical_warfare_I}. The book \cite{Lueck:Big_book} deals specifically with von Neumann dimensions in the context of group actions (Chapters 1 and 2). The introductory article \cite{Eckmann_intro} also covers von Neumann dimensions in as much as we need.


\section{Decidable $l^2$-zero-divisor problem}\label{section-sofic-decidable}
Consider the following computational problem.
\begin{problem}{$l^2$-zero-divisors-in-$\Z[G]$}
\textbf{Input:} $M\in \Z[G]$

\textbf{Problem:} Is $\ker M =\{0\}$?
\end{problem}

Although this problem is potentially easier than Kernel-over-$\Z[G]$, the proof of decidability is the same. Because of simpler notation, we only show that  $l^2$-zero-divisors-in-$\Z[G]$ is decidable, and leave the general case of Proposition \ref{prop-main} for the reader.

\begin{prop}\label{prop-main-honestly}
  Let $G$ be a finitely-generated sofic BFS-group with a decidable word problem and for which the Atiyah conjecture holds. Then $l^2$-zero-divisors-in-$\Z[G]$  is decidable.
\end{prop}

The property of sofic groups which we use is as follows. We defer the proof to the appendix.

\begin{lemma}\label{lemma-sofic-property}
 Let $G$ be a sofic group. There is a computable function $h=h_G\colon \N_+ \to \N_+$ such that if $T\in \Z[G]$ is positive and self-adjoint, its support consists of at most $n$ elements, and coefficients of $T$ are bounded by $n$, then
$$
|\dimvn\ker T - \trvn (1-\frac{T}{\|T\|_1})^{h(n)}| <\frac{1}{n},
$$
where $\|T\|_1$ is the sum of absolute values of coefficients of $T$.
\end{lemma}

\begin{remarks}
 \begin{enumerate}
  \item The role of the denominator is to make sure that the operator norm of $\frac{T}{\|T\|_1}$ is at most $1$.  
  \item The convergence of the sequence $ \tr (1-\frac{T}{\|T\|_1})^n$ to $\dimvn\ker T $ is true for any group - it is a consequence of the spectral theorem.
  \item In the case of sofic groups we get a  function $h$, which \textit{does not depend on the group $G$} (see Proposition \ref{prop-concrete-bound}). It would be interesting to show that for every (possibly non-sofic) group $G$ there exists some computable function $h$ as above.
 \end{enumerate}
\end{remarks}

\begin{proof}[of Proposition \ref{prop-main-honestly}]
Let $k = k(G)\in \N$ be the number guaranteed by the Atiyah conjecture, i.e.. for any $T\in \Z [G]$ we have $\dimvn\ker T \in \{0,\frac{1}{k}, \frac{2}{k},\ldots\}$.

 Let $T\in \Z [G]$. Note that $T$ is an $l^2$-zero divisor if and only if the spectral measure of $T$ has an atom at $\{0\}$. If the measure of an atom is positive then, by the Atiyah conjecture, it is at least $\frac1{k}$.

 The algorithm should compute $\trvn (1-\frac{T}{\|T\|_1})^{h(3k)}$ (this is possible because the word problem is decidable). Let the outcome of this computation be called $c$. By the definition of $h$ we have  $|c- \dimvn \ker T| < \frac1{3k}$. Since we know that $\dimvn\ker T\neq 0$ is equivalent to $\dimvn\ker T \ge \frac1{k}$, we get that $T$ is a zero-divisor if and only if $c>\frac1{3k}$.
\end{proof}

Recall that the Kaplansky's zero-divisor conjecture states that $0$ is the only zero-divisor in $\Z[G]$ when $G$ is torsion-free. In the case of a torsion-free group with a decidable word problem, obtaining an algorithm for Zero-divisors-in-$\Z[G]$ can be seen as a ``weak version'' of the zero-divisor conjecture, since if the latter holds then  $T\in \Q[G]$ is a zero-divisor if and only if $T=0$, and using decidability of the word problem we can decide if $T =0$.  

However, our results give very little new information about the zero-divisor conjecture. We only mention the following corollary to Proposition \ref{prop-main}. Its main interest comes from the fact that it is often difficult to establish the zero-divisor conjecture for extensions.

\begin{cory}\label{intro_prop_zero_div_conjecture}
 Let $G$ be a finitely-generated  amenable torsion-free group with decidable word problem for which the zero-divisor conjecture holds. Then Zero-divisors-in-$\Z[G_1]$ is decidable for any $G_1$ which contains $G$ as a subgroup of finite index.
\end{cory}
\smallskip\textit{Sketch of proof.}
The zero-divisor conjecture is known to imply the Atiyah conjecture in the case of amenable groups. Indeed, as mentioned in the introduction, the zero-divisor conjecture implies the $l^2$-zero-divisor conjecture, which is equivalent to the statement that all non-zero elements of $\Q[G]$ are invertible in the ring of affiliated operators. As such we can use a ring-theoretic localization (\cite[Example 8.16]{Lueck:Big_book}) to obtain an embedding of $\Q[G]$ into a skew-field of affiliated operators, which by results of Linnell (\cite{Linnel:Atiyah_conjecture_for_free_groups}) is equivalent to the Atiyah conjecture.

After passing to a subgroup of $G$, we can assume that $G$ is normal in $G_1$. Let $m=[G_1:G]$. Using coset representatives, it is not difficult to show that  we can embed $\Q[G_1]$ into $M_m(\Q[G])$ in a way which scales the von Neumann trace by $m$. Since the Atiyah conjecture holds for $G$, we obtain that $\cal C(G_1)\subset \{0,\frac{1}{m},\frac{2}{m}\ldots\}$, i.e.  (a weak form of) the Atiyah conjecture for $G_1$. 

The decidability of the word problem for $G_1$ follows from explicitly writing down the embedding of $G_1$ into $M_m(\C[G])$. Finally, amenable groups are sofic (\cite{Pestov_Hyperlinear_and_sofic_groups_a_brief_guide}), and so we can apply Proposition \ref{prop-main-honestly}. {\hfill$\Box$}
\medskip

The scope of the above corollary is rather limited: if we additionally assume that $G$ is \textit{elementary} amenable  then $G_1$ also is, and in this case if $G_1$ is torsion-free then the zero-divisor conjecture holds for $G_1$ (e.g. by \cite{Linnel:Atiyah_conjecture_for_free_groups}).

We finish by pointing out another corollary of Proposition \ref{prop-main}. It is possible that it could be also established by analysing the proof of the Atiyah conjecture for elementary amenable groups in \cite{Linnel:Atiyah_conjecture_for_free_groups}. By Corollary \ref{cory-intro-zero-divisors-in-lamplighter-4}, the statement is false without the assumption that $G$ is a BFS-group.

\begin{cory}\label{intro_cory_elementary_amenable_decidable}  Let $G$ be an elementary amenable BFS-group. Then the decidability of the word problem implies the decidability of the zero-divisor problem.
\end{cory}
\begin{proof} By \cite{Linnel:Atiyah_conjecture_for_free_groups}, the Atiyah conjecture holds for $G$.  The corollary follows because  for an amenable group we have that $M\in M_k(\C[G])$ is a zero-divisor if and only if $\ker M \neq \{0\}$.
\end{proof}

\section{Preliminaries on Turing machines}\label{section_prelim_on_turing_machines}

For information on Turing machines see \cite{sipser_introduction_to_the_theory_of_computation}. Given a Turing machine $M$ we denote by $A(M)$ and $S(M)$ the alphabet and the set of states of $M$. The rules by which the Turing machine operates are referred to as the \textit{transition table}. We assume that there are three special states \textsc{initial}, \textsc{reject} and \textsc{accept} in $S(M)$ and that the transition table is such that the state \textsc{initial} cannot be entered from any other state, and it is left in the first step of operation. 

For the standard Turing machines (i.e. where there is only one tape and where the tape head can both write and read the symbols on the tape), we assume that the tape can contain a special symbol \textsc{empty} which is not an element of  $A(M)$, corresponding to an empty place on the tape. The transition table has to specify the behaviour of the machine upon reading the \textsc{empty} symbol.

We will also consider read-only Turing machines. For these, instead of the symbol \textsc{empty}, we demand the existence of the symbol \textsc{delimiter} which also is not an element of $A(M)$. We assume that the transition table is such that whenever a tape head moves left and afterwards reads the symbol \textsc{delimiter}, its next move cannot be to the left; similarly for ``right'' in place of ``left''.

Finally, we will consider read-only Turing machines with multiple tape heads and with multiple tapes. For the former, the instructions in the transition table cannot be conditioned on whether two tape-heads are in the same position. 

{\it Configuration} of a  Turing machine $M$ is a triple consisting of a tape (or multiple tapes) with symbols written on it, a position of the tape head (or tape heads), and a state of $M$. {\it Initial configuration}  is a configuration whose state is the \textsc{initial} state, and (i) in the case of a read-write Turing machine the tape consists of infinitely many \textsc{empty} symbols, followed by a word in the alphabet $A(M)$, followed by infinitely many \textsc{empty} symbols; (ii) in the case of a read-only Turing machine each tape is finite and has the form \textsc{delimiter}, followed by the word in $A(M)$, followed by \textsc{delimiter}. In both cases the tape head (resp. heads) is assumed to be on the first symbol belonging to $A(M)$ on the tape (resp. each tape).

$M$ is {\it foolproof} if it halts (i.e. enters the \textsc{accept} or the \textsc{reject} state) no matter what configuration it is put into before it starts operating. Note that this is stronger than saying that $M$ always halts, as we require that $M$ halts also on configurations which are not initial  (and which in principle might not appear in any computation which starts from an initial configuration).

We start with  a variant of a folklore proposition about read-only Turing machines (\cite[Exercise 5.26]{sipser_introduction_to_the_theory_of_computation}).
 
\begin{proposition}\label{prop_machine_2_read_only} 
There is an algorithm which given a Turing machine $M$ produces a read-only Turing machine $\cal R(M)$ with two tape heads such that there is a word which $M$ accepts if and only if there is a word which $\cal R(M)$ accepts.
\end{proposition}
\smallskip\textit{Sketch of proof.}
Let us describe  $\cal R(M)$ explicitly. $A(\cal R(M))$ consists of the symbol \textsc{next configuration}, the elements of $S(M)$ and the elements of $A'(M):= (A(M)\cup{\textsc{empty}})\times\{0,1\}$. 

The machine $\cal R(M)$ should first check whether the input starts with \textsc{next configuration}, followed by a symbol from $S(M)$, followed by a word in $A'(M)$, followed by \textsc{next configuration}, followed by... and finishing with \textsc{next configuration}. Afterwards it should check if each word in $A'(M)$ has precisely one symbol which maps to $1$ (which should be interpreted as the position of the tape head of $M$). For all this we need just one tape head.

In the second stage  $\cal R(M)$ checks whether the consecutive configurations on the tape indeed correspond to the consecutive configurations of the execution of $M$. This can be done with two tape heads.

In the final stage $\cal R(M)$ should check whether the \textsc{accept} symbol appears in the input. {\hfill$\Box$}

\begin{cory} \label{cory_machine_2_read_only_multitape} There is a an algorithm which given a Turing machine $M$ produces a foolproof read-only Turing machine $\cal F(M)$ with three tapes, such that there is a word which $M$ accepts if and only if there is a word which $\cal F(M)$ rejects.
\end{cory}

\smallskip\textit{Sketch of proof.}
Let $\cal R'(M)$ be the machine $\cal R(M)$  from the proposition with exchanged accepting and rejecting states. The first step is constructing a machine $\cal R''(M)$ with two tapes which simulates $\cal R'(M)$ by first checking whether both tapes have the same words written on them.

To assure the foolproofness we add the third tape, and we call its tape head $H_3$. For each pair of states $\si, \tau \in S(\cal R''(M))$ we add a new state $D(\si,\tau)$. If the transition table of $\cal R''(M)$ for a state $\si\in S(\cal R''(M))$ and symbols $s_1, s_2\in A(\cal R''(M))$ on the consecutive tapes requires changing the state to $\tau$, then in the transition table for $\cal F(M)$ we require changing the state to $D(\si, \tau)$. 

In each of the states $D(\si, \tau)$, the machine $ \cal F(M)$ moves $H_3$ to the right, and moves to the accepting state if $H_3$ reaches the end of the input, and to $\tau$ otherwise. In this way $\cal F(M)$ simulates $\cal R''(M)$ for the number of steps which is equal to the number of symbols between \textsc{delimiters} on the third tape, and terminates afterwards.{\hfill$\Box$}

\section{Embedding a Turing machine in a group ring}\label{section_embed}

A more detailed example of associating an element of a group ring to a Turing machine is in \cite[Section 5]{arxiv:grabowski-2010}. We quote some definitions from there.

Let $(X, \mu)$ be a probability measure space and $\rho:\Ga\curvearrowright X$ be a  measure preserving action of a countable discrete group $\Ga$ on a probability measure space $X$.  A {\it dynamical hardware} is the following data: $(X,\mu)$, the action $\rho$, and a division $X = \bigcup_{i=1}^n X_i$ of $X$ into disjoint measurable subsets. We denote such a dynamical hardware by $(X)$.

Suppose now that we are given $(X)$ and we choose three additional distinguished disjoint subsets of $X$, each of which is a union of certain $X_i$'s: the initial set $I$, the rejecting set $R$, and the accepting set $A$ (all or some of them might be empty). Furthermore, suppose that for every set 
$X_i$, we choose one element $\ga_i\in \Ga$ in such a way that the elements corresponding to the sets $X_i$ which are subsets of $R\cup A$ are equal to the neutral element  of $\Ga$.  A {\it dynamical software} for  $(X)$ is the following data: the distinguished sets $I, A$ and $R$ and the choice of elements $\ga_i$.

We define a map $T_X: X\to X$ by
$$
	T_X(x):= \rho(\ga_i)(x) \quad \text{ for } x\in X_i.
$$
The whole dynamical software will be denoted by $(T_X)$. A {\it Turing dynamical system} $(X, T_X)$ is a dynamical hardware $(X)$ together with a dynamical software $(T_X)$ for $(X)$. 

If $x\in X$ is such that for some $k$ we have $T_X^k(x)= T_X^{k+1}(x)$ then we define $T_X^\infty(x):= T_X^k(x)$. Otherwise we leave $T_X^\infty$ undefined.

The {\it fundamental set}  of $(X,T_X)$ is the subset $\cal F_1(T_X)$ of $I$ consisting of all those points $x$ such that $T^\infty(x) \in A$ and for no point $y\in X$ one has $T_X(y)=x$.  It is measurable (\cite{arxiv:grabowski-2010-2}), and therefore we define $\Om_1(T_X)$, the {\it fundamental value} of $(X,T_X)$, to be equal to $\mu(\cal F_1(T_X))$.

We  say that $(X,T_X)$ {\it stops on any configuration}, if for almost all $x$ we have $T_X^\infty(x)\in A\cup R$;  it {\it has disjoint accepting chains}, if for almost all $x\in \cal F_1$ we have that for all $y\in \cal F_1$  the inequality $T^\infty(x) \neq T^\infty(y)$ holds; finally it {\it does not restart}, if $\mu(T_X(X)\cap I)=0$. 

Suppose now that $(X,\mu)$ is a compact abelian group with the normalized Haar measure and the action of $\Ga$ is by continuous group automorphisms. Let $\wh X$ be the Pontryagin dual of $X$  and let us consider the rational group ring $\qring {\wh X}$ as a subring of $L^\infty (X)$ through the Pontryagin duality. Suppose that the characteristic functions $\chi_i$ of the sets $X_i$ are elements of $\qring{\wh X}$.

Let $G:= \wh X \rtimes_{\wh{\rho}} \Ga$ and define $T,S\in \qring G$ by $T:= \sum_{i=1}^n \ga_i\chi_i$ and
$$
S:=  (T+ \chi_X - \chi_I - \chi_A - \chi_R)^*(T+ \chi_X - \chi_I - \chi_A - \chi_R) + \chi_A.
$$

The following theorem is \cite[Theorem 4.3]{arxiv:grabowski-2010-2}.
\begin{theorem}\label{thm_external}
If $(X, T_X)$ stops on any configuration, has disjoint accepting chains, and does not restart, then $\dimvn\ker S$ is equal to $\mu(I) - \Om_1(T_X)$.
\end{theorem}

The main result of this section is the following proposition.
\begin{proposition}\label{prop_embed_turing_in_group_ring}
There is an algorithm which given a foolproof read-only Turing machine $M$ on three tapes produces a finite group $H(M)$ and an element $T(M)$ of $\zring{(\lamplighter)^3 \times H(M)}$ in such a way that there exists a word  which $M$ rejects if and only if  $\dimvn\ker T(M) \neq 0$.
\end{proposition}

\begin{proof}The first step is to algorithmically associate to $M$ a Turing dynamical system which fulfils the conditions of Theorem \ref{thm_external}, in such a way that $\Om_1(T_X)$ is smaller than $\mu(I)$ if and only if there is a word which $M$ rejects. 

Let $n$ be the smallest natural number such that $|S(M)|<2^n$. Consider the group $\cal S:=\ckl{2}^n$. Let us choose any $|S(M)|$ non-zero elements of it and label them with the elements of $S(M)$. For any pair of states $\si, \tau\in S(M)$ we fix an automorphism $\ga(\si, \tau)\in \aut(\cal S)$ which sends $\si$ to $\tau$. 

Let $m$ be the smallest natural number such that $|A(M)| + 1 \le 2^m$ and let $\cal A := \ckl{2}^m$. Let us choose any $|A(M)|$ non-zero elements of $\cal A$ and label them with the elements of $A(M)$. Let all the other elements of $\cal A$ be labelled with the \textsc{delimiter} symbol. 

Let us define a dynamical hardware. For the compact abelian group we take $X= (\prod_{\ckl{}} \cal A)^3\times \cal S$. For $\Ga$ take $\ckl{}^3\times \aut(\cal S)$. Each coordinate of $\ckl{}^3$ acts by shifting the appropriate coordinate of $(\prod_{\ckl{}} \cal A)^3$, and $\aut(\cal S)$ acts in the natural way on $\cal S$. 

We use the following notation for the cylinder subsets of $X=(\prod_{\ckl{}} \cal A)^3 \times \cal S$, motivated by thinking about the points of $X$ as the configurations of a Turing machine with the alphabet $\cal A$ and the set of states $\cal S$. The set $\{[(x_i,y_i,z_i),s]\in X\colon x_0=a, y_0=b, z_0=c,s=\si\}$ is denoted by 
$$
\left[ \begin{array}{l}\underline{a}\\\underline{b}\\\underline{c}\end{array}\right][\si], 
$$
the set $\{[(x_i,y_i,z_i),s]\in X\colon x_0=a,x_{-1}=a', y_0=b, y_{-1}=b', z_0=c, z_{-1}=c',s=\si\}$ is denoted by 
$$
\left[\begin{array}{l}{a'}\,\underline{a}\\{b'}\,\underline{b}\\{c'}\,\underline{c}\end{array}\right][\si],
$$
and so on. The set $\{[(x_i,y_i,z_i),s]\in X\colon s=\si\}$ is denoted by $[][\si]$. 

To finish the description of the dynamical hardware we need to specify a division of $X$. We start with the division
$$
X = \bigsqcup_{a,b,c\in \cal A, \si \in \cal S} \left[ \begin{array}{l}\underline{a}\\\underline{b}\\\underline{c}\end{array}\right][\si].
$$
and replace each  $\left[ \begin{array}{l}\underline{a}\\\underline{b}\\\underline{c}\end{array}\right][\textsc{initial}]$  by two sets: 
$$
 \left[ \begin{array}{l}\textsc{delimiter}\,\underline{a}\\\textsc{delimiter}\,\underline{b}\\\textsc{delimiter}\,\underline{c}\end{array}\right][\textsc{initial}]
$$
and its complement. 

It is a standard calculation using the Pontryagin duality to check that the characteristic functions of the sets above are all elements of $\qring {\wh{X}}$.

Now we define a dynamical software for $(X)$. The Initial set $I$ is the union of
$$
 \left[ \begin{array}{l}\textsc{delimiter}\,\underline{a}\\\textsc{delimiter}\,\underline{b}\\\textsc{delimiter}\,\underline{c}\end{array}\right][\textsc{initial}]
$$
over all $a,b,c\in A(M)$, the Accepting set is defined to be $A=[][\textsc{accept}]$ and the Rejecting set is the union of $[][\textsc{reject}]$ and all the sets $[][\si]$, where $\si\in \cal S-S(M)$.

The assignment of elements of $\Ga$ is as follows. On the set 
$$\left[ 
\begin{array}{l}\underline{a}\\\underline{b}\\\underline{c}\end{array}\right]
[\si],
$$
where  $a,b,c\in\cal A$ and $\si\in S(M)$, we act with $\ga(\si,\tau)$, where $\tau$ is such that the transition table of $M$ requires changing the state to $\tau$ when being in the state $\si$ and encountering the symbols $a,b,c$ on the first, second and third tape. Everywhere else, i.e. when $\si\in\cal  S\smallsetminus S(M)$, we act with the identity element of $\Ga$. 

The system just defined stops on any configuration because $M$ is foolproof. It does not restart because of the standing assumption on Turing machines that it is impossible to enter the state \textsc{initial} and that the state \textsc{initial} is left immediately.

Finally $(X,T_X)$ has disjoint accepting chains because of the following observation. Let $x$ be an element of the fundamental set. In particular on all three tapes it is of the form  $(\textsc{delimiter}\, \underline{a})$. Since with probability $1$ there is another \textsc{delimiter} sign on each tape, we can as well assume that $x$ is of the form $(\textsc{delimiter} \,\underline{a}\ldots \textsc{delimiter})$, with no delimiter symbol in ''$\ldots$``. The subsequent iterations of $T_X$ evaluated on $x$ ''move the underlinings`` on the tapes, but they cannot move them beyond the delimiter symbols. Thus whatever is the image of $x$ in the accepting set, we can recover from this image the original configuration $x$. This means, that the accepting chains are disjoint.

Let $H(M):= \ckl{2}^m \rtimes \aut(\ckl{2}^m)$, where the semi-direct product is with respect to the Pontryagin-dual action. The previous theorem gives us an operator $T(M) \in \qring {(\ckl{2}^n\wre\ckl{})^3\times H(M)}$ such that  $\dimvn\ker T(M) = \mu(I) - \Om_1(T_X)$. However, since $(X,T_X)$ stops on any configuration, the right-hand side is precisely the measure of those points in  
$$\left[ \begin{array}{l}\textsc{delimiter}\,\underline{a}\\\textsc{delimiter}\,\underline{b}\\\textsc{delimiter}\,\underline{c}\end{array}\right][\textsc{initial}]
$$ 
which are mapped by some iteration of $T_X$ to $[][\textsc{reject}]$. Given a word which $\cal M$ rejects, we can produce a set of positive measure of such points. Conversely, a set of positive measure of such points must contain a configuration which on each tape contains \textsc{delimiter} symbols both to the left and to the right of the underlined symbols, and so we can produce from it an input which is rejected by $\cal F(M)$. 

Finally we note that $\ckl{2}^n\wre\ckl{}$ is isomorphic (in an algorithmic fashion with respect to $n$) to a subgroup of $\lamplighter$. The proposition follows after clearing the denominators in $T(M)$.
\end{proof}

We are now ready to prove our main result, Theorem \ref{thm-main-honestly}. We restate it for reader's convenience.

\begin{thm*}
The Kernel-over-$\Z[G]$ problem is undecidable for $G = (\lamplighter)^3$.
\end{thm*}

\begin{proof}
We show that if this was not the case then we could produce an algorithm which given a Turing machine $M$ decides whether there exists an input which $M$ accepts. The latter problem is well-known to be undecidable (see \cite{sipser_introduction_to_the_theory_of_computation}).

Starting with $M$, we can algorithmically produce the read-only foolproof machine $\cal F(M)$ on three tapes from Proposition \ref{prop_machine_2_read_only} with the property that there exists a word which $M$ accepts if and only if there exists a word which $\cal F(M)$ rejects.

Now, thanks to Proposition \ref{prop_embed_turing_in_group_ring} we can algorithmically produce a finite group $H(\cal F(M))$ and 
$$T(\cal F(M)) \in \zring{ (\lamplighter)^3\times H(\cal F(M))}$$ 
such that $\dimvn\ker T(\cal F(M))\neq 0$ if and only if there exists a word which $\cal F(M)$ rejects.

But note that $\zring{ (\lamplighter)^3\!\times\! H(\cal F(M))}$ is isomorphic  to $\zring{ (\lamplighter)^3}\otimes \zring{ H(\cal F(M))}$. Furthermore, the von Neumann trace corresponds to the product of von Neumann traces. However, $\zring {H(\cal F(M))}$ can be algorithmically embedded into $M_k(\Z)$  in a von Neumann trace-preserving fashion, by using the left regular representation. Therefore we get an embedding $j_M\colon\zring{ (\lamplighter)^3}\otimes \Z[ H(\cal F(M))] \to M_k(\Z[\lamplighter)^3])$ which preserves von Neumann traces. The latter property implies that $\dimvn\ker S = 0 \iff \dimvn\ker j_M(S) = 0$ (see e.g. \cite[Lemma 1.9]{arxiv:grabowski-2010}). 
%
%
%
%
%

To recap, we have algorithmically produced an element $j_M(T(\cal F(M))) \in M_k(\Z[\lamplighter)^3])$ such that there exists a word which $M$ accepts if $\dimvn\ker j_M(T(\cal F(M)))\neq 0$. This ends the proof. 


\end{proof}

%
%

\section{Zero-divisor problem for $(\lamplighter)^4$}\label{section-lamplighter4}

We finish the article by proving the following corollaries.

\begin{cory}\label{cory-zero-divisors}
 Zero-divisors-in-$\Z[(\lamplighter)^4]$ is undecidable.
\end{cory}
It is of some interest to have finitely presented examples, so we point out the following corollary. 

\begin{cory}\label{cory-zero-divisors-fp}
 Let $G$ be a group given by the presentation 
$$
	\langle a,t,s\,|\,a^2=1,[t,s]=1,[t^{-1}at,a]=1,s^{-1}as=at^{-1}at\rangle.
$$
Then $G^4$ is metabelian, all torsion elements are of order $2$, its word problem is decidable, but Zero-divisors-in-$\Z[G^4]$ is undecidable.
\end{cory}

\begin{remarks}
Recall that the torsion problem for $G$ is the algorithmic problem whose input is a word $w$ in the generators and the question is whether $w$ represents an element of finite order in $G$.

Note that the decidability of Zero-divisors-in-$\Z[G]$ implies the decidability of the torsion problem, because $g\in G$ is a torsion element if and only if  $1-g\in \zring{G}$ is a zero-divisor.  Therefore, the corollaries above are interesting only when we note that the torsion problem is decidable for $(\lamplighter)^4$ and $G^4$.
\end{remarks}

\begin{proof}[of Corollary \ref{cory-zero-divisors}]
The matrix algebra $M_n(\Z[G])$ is naturally isomorphic to $\Z[G]\otimes M_n(\Z)$. Under this isomorphism the von Neumann trace $\trvn$ corresponds to $\trvn \otimes \tr$.

\smallskip\textit{Claim.}\label{lemma_embed_matrix_rings} There exists an algorithm which given $n$ produces an embedding $i_n$ of the matrix algebra $M_n(\Z)$ into the group ring $\Q[\lamplighter]$ such that  $\frac{1}{2^{n+2}} \tr_{vN}\circ i_n = \tr$ and such that $i_n$ preserves taking adjoints. 

\smallskip\textit{Proof of Claim.} Let $\chi_j \in\qring\lamplighter$ be the Fourier transform of the characteristic function of $[0\, 1^{j-1}\, \underline{1}\, 1^{n-j}\, 0]$ (see the previous section for the notation). Let $E_{ij} := t^{i-j}\chi_j\in 
\qring\lamplighter$. It is enough to check that $E_{kl}\cdot E_{ij} = \de_i^l \cdot E_{kj}$. Note
$$
E_{kl}\cdot E_{ij} = t^{k-l}\chi_l \cdot t^{i-j}\chi_j = t^{k-l}\chi_l \cdot \chi_i t^{i-j},
$$
which is non-zero only if $i=l$; in this case it is equal to
$$
t^{k-i}\chi_i t^{i-j} = t^{k-i}t^{i-j} \chi_j =t^{k-j} \chi_j,
$$
as claimed. The statement about the traces follows by noting that the measure of the set $[0\, 1^{j-1}\, \underline{1}\, 1^{n-j}\, 0]$ is equal to $\frac{1}{2^{n+2}}$.{\hfill$\Box$}

\smallskip
Altogether, there is an algorithm which for a given $n$ produces an embedding $j_n$ of $M_n(\Z[(\lamplighter)^3])$ into $\Z[(\lamplighter)^4]$. Furthermore $j_n$ scales the von Neumann trace and preserves adjoints. Using \cite[Lemma 1.9]{arxiv:grabowski-2010} we deduce that $T\in M_n(\Z[(\lamplighter)^3])$ is an $l^2$-zero-divisor if and only if $j_n(T)\in\Z[(\lamplighter)^4]$ is an $l^2$-zero-divisor. 

Since $\lamplighter$ is amenable, in both algebras being a zero-divisor is equivalent to having non-trivial $l^2$-kernel, and the corollary follows.
\end{proof}

\begin{proof}[of Corollary \ref{cory-zero-divisors-fp}]
All the properties follow from \cite{Grigorchuk_Linnel_Schick_Zuk}. In particular it is proven there that $\lamplighter$ embeds into $G$ (similar embeddings of wreath products were considered earlier by Baumslag, see for example \cite{baumslag-a-finitely-presented-metabelian-group-with-a-free-abelian-derived-group-of-infinite-rank}).
%
\end{proof}

\appendix
\section{Sofic groups and a computable bound on spectral densities}
Sofic groups were introduced in  \cite{gromov_endomorphisms_of_symboliv_algebraic_varieties}. The article  \cite{Pestov_Hyperlinear_and_sofic_groups_a_brief_guide} is a very readable survey. 

We start with some notation. Given a graph $K$, the set of vertices of $K$ is denoted by $V(K)$. 
The Hilbert space spanned by $V(K)$ is denoted by $l^2(K)$.
The elements of the standard basis of $l^2(K)$ are denoted by $\zeta_v$, $v\in V(K)$.  The ball of radius $R$ at $v\in V(K)$ is denoted by $B_K(v,R)$.

If $K$ is oriented and edge-labelled by complex numbers then the associated convolution operator is the unique operator $T\colon l^2(K) \to l^2(K)$ such that $\langle T(\zeta_x), \zeta_y\rangle$ is equal to $0$ if there are no edges between $x$ and $y$, and to the sum of all the labels of edges from $x$ to $y$ otherwise.

If $G$ is a group, and $(g_i) = (g_1,\ldots, g_n)$ is a symmetric sequence of elements of $G$ (i.e for every $g\in G$ the number of times $g$ appears is the same as the number of times $g^{-1}$ appears) which generates $G$, the Cayley diagram of $G$ with respect to $(g_i)$, denoted $\cal C(G, g_i)$, is the oriented labelled graph whose vertices are elements of $G$ and with an oriented edge with label $g_i$ from $a$ to $b$ if $ag_i=b$.

We will now define sofic groups. Let $G$  be a finitely generated group, and $(g_1, \ldots, g_n)$  a symmetric generating sequence.  Let $K$ be a finite oriented graph  edge-labelled by the sequence $(g_i)$ in such a way that 
\begin{enumerate}
\item at every vertex the out- and in-degrees are at most $n$ and each of the symbols $g_i$ appears at most once as the out- and at most once as the in-label; 
\item the number of edges labelled by a given $g_i$ from $v$ to $w$ is equal to the number of such edges from $w$ to $v$ (and by the previous assumption it is either zero or one). 
\end{enumerate}

For $\eps\ge 0$, $R>0$, we say that $K$ is an $(\eps, R)$-sofic approximation of $G$ with respect to the sequence $(g_i)$ if the set 
$$
\{v\in V(K): B_K(v,R) \text{ is isomorphic to } B_{\cal C(G, g_i)}(e, R)\}
$$
has at least $(1-\eps)\cdot |V(K)|$ elements. The isomorphism is meant in the sense of edge-labelled graphs.

The group $G$ is \textit{sofic} if for every $\eps$ and $R$ there exists an $(\eps, R)$-sofic approximation of $G$ with respect to $(g_i)$ (this definition does not depend on the choice of a generating sequence). A general countable group is sofic if all its finitely generated subgroups are sofic.

Suppose $T\in \C [G]$ can be written as $\sum_{i=1}^n a_ig_i\in \C [G]$.  We define the graph $K(T)$ by taking $K$ and replacing each label $g_i$ by $a_i$. Note that $K(T)$ depends on the choice of the representation for $T$ as a sum, not only on $T$, but this will not lead to any problems. 

Let  $\pi_K(T)\colon l^2(K)\to l^2(K)$ be the convolution operator on $K(T)$. Note that if $\sum_i |a_i| \le 1$ then the operator  norm of $\pi_K(T)$ is at most $1$.

\smallskip
%

 The author learned about the following proposition from Andreas Thom. Since it does not seem to be in the literature (but compare \cite{thom_sofic_groups_and_diophantine_approximation}, \cite{Elek_Szabo:Hyperlinearity_eseentially_free_actions_and}, and the proof of Lemma 3.179 in \cite{Lueck:Big_book}), we give a proof based on the proof of the L\"{u}ck's approximation theorem (the latter originally proven in \cite{lueck_approximating_l2_invariants_by_their_finite_dimensional_analogues}). 

\begin{proposition}\label{prop-concrete-bound}
Let $G$ be a sofic group and let $T\in \qring G$ be a positive self-adjoint element whose sum of coefficients is smaller than $1$. Suppose that the smallest common multiple of the denominators of the coefficients of $T$ is equal to $C$. Then 
$$
 |\tr ((1-T)^n) - \mu_T(\{0\})| < \frac{3 C}{\log(n)}.
$$
\end{proposition}

\begin{proof}
 Let $\eps(n) = \frac{C}{\log(n)}$. Note that $\eps(n^k)=\frac{\eps(n)}{k}$.  Let $T=\sum_{i=1}^n a_i g_i$; without any loss of generality we can suppose that $G$ is generated by the symmetric sequence $(g_1,\ldots, g_n)$. Let $K_n$ be an $(\eps(n),n)$-sofic approximation of $G$ for the sequence $(g_i)$, and let $\pi_n (T)= \pi_{K_n}(T)$ be the associated convolution operator on $l^2(K)$. Let $\mu_n$ be the (normalized) spectral measure of $\pi_n(T)$ and let $\mu$ be the spectral measure of $T$.

In this proof all integrals are over the interval $[0,1]$, unless explicitly stated otherwise.

\medskip\textit{Claim.} For all $m,n$ we have $\mu_m( (0,\frac{1}{n})) < \eps(n)$. 

\textit{Proof of Claim.} Consider the characteristic polynomial of the matrix of $\pi_m(T)$ in the standard basis, divided by the monomial $X^f$, where $f$ is the multiplicity of the eigenvalue $0$. Its coefficients are rational numbers with denominators at most $C^{|K_m|}$, and its roots are precisely the non-zero eigenvalues of $\pi_m(T)$. Therefore
$$
| \prod \al| \ge \frac{1}{C^{|K_m|}},
$$
where the product is over non-zero eigenvalues of $\pi_m(T)$. By estimating the roots smaller than $\frac{1}{n}$ by $\frac{1}{n}$, and the rest by $1$, we get 
$$
\frac{1}{C^{|K_m|}} \le \left( \frac{1}{n} \right)^{\mu_m( (0,\frac{1}{n})) \cdot |K_m|},
$$
and so
$$
\frac{1}{C} \le \left( \frac{1}{n} \right)^{\mu_m( (0,\frac{1}{n}))},
$$
from which the claim follows by taking logarithms.{\hfill$\Box$}\smallskip

\medskip\textit{Claim.} For all $n$ we have $\mu((0,\frac{1}{n}))<\eps(n)$.

\textit{Proof of Claim.} We first show that for any continuous function $f$ we have 
$$
\lim_{m\to\infty} \int f(x)\,d\mu_m(x) = \int f(x)\,d\mu(x),
$$
i.e. that the measures $\mu_m$ converge weakly to $\mu$. By Weierstrass approximation it is enough to show it for a monomial $x^n$. In this case we have
$$
|\int x^n\,d\mu(x) - \int x^n\,d\mu_m(x)| = |\tr_{vN} T^n - \frac{1}{|V(K_m)|}\tr \pi_m (T)^n|, 
$$
and the right hand side is equal to 
$$
| \langle T^n\zeta_e, \zeta_e\rangle - \frac{ \sum_{v\in V(K_m)} \langle \pi_m(T)^n\zeta_v,\zeta_v\rangle}{|V(K_m)|}|.
$$
Suppose $m>n$. Then apart from $\eps(m)\cdot |V(K_m)|$ vertices, we have $\langle T^n\zeta_e, \zeta_e\rangle = \langle \pi_m(T)^n\zeta_v,\zeta_v\rangle$, since both quantities depend only on the ball of radius $n$ around the relevant vertex. Noting this and estimating both $\langle T^n\zeta_e, \zeta_e\rangle$ and $\langle \pi_m(T)^n\zeta_v,\zeta_v\rangle$ by $1$ we get that the above is smaller or equal to $2\cdot \eps(m)$.

The claim now follows from the weak convergence -  let $\de$ be such that $\mu([\de, \frac{1}{n}-\de])$ is almost equal to $\mu((0,\frac{1}{n}))$, and take $f$ to be a continuous approximation of the characteristic function of $(0,\frac{1}{n})$ such that $f(0)=f(\frac{1}{n})=0$,  $f(x)=1$ for $x\in [\de, \frac{1}{n}-\de]$, and $f$ is linear on $(0,\de)$ and $(\frac{1}{n}-\de,\frac{1}{n})$.{\hfill$\Box$}
\smallskip
 
We finally show that  $|\tr_{vN} (1-T)^n - \dimvn\ker \pi(T)| < 3\cdot \eps(n).$ The left hand side is equal to
$$
  \int_{(0,1]} (1-x)^n d\mu(x) = \int_{(0,\frac{1}{\sqrt{n}})} (1-x)^n d\mu(x)+ \int_{[\frac{1}{\sqrt{n}},1]} (1-x)^n d\mu(x)
$$
We estimate the first integrand by $1$, and the second by $(1-\frac{1}{\sqrt{n}})^n$ to get that the above is equal to at most 
$$
\mu((0,\frac{1}{\sqrt{n}})) + (1-\frac{1}{\sqrt{n}})^n,
$$
which by the previous claim and a simple calculation is not bigger than
$$
\eps(\sqrt{n}) + (\frac{2}{e})^{\sqrt{n}} < 3\cdot\eps(n).
$$
\end{proof}



\bibliographystyle{alpha}
\bibliography{bibliografia}
\end{document}